\documentclass[12pt]{amsart}
\usepackage{amsfonts,amsmath,amssymb,amsthm}
\usepackage{color,graphicx,subfigure}
\usepackage[all]{xy}
\usepackage{enumerate}
\usepackage{hyperref}
\hypersetup{
    colorlinks=true,
    linkcolor=cyan,
    citecolor=red,
    urlcolor=blue,
}
\theoremstyle{plain}
\newtheorem{thm}{Theorem}
\newtheorem{lma}{Lemma}
\newtheorem{crl}{Corollary}
\newtheorem{cnj}{Conjecture}

\title{Perturbation Effects on Word Lengths in Three-Reflection Symmetric Presentations of Dihedral Groups}

\author{Michael P. Allocca} 
\address{Department of Mathematics, Computer Science, and Statistics
Muhlenberg College, 
Allentown, PA, 18104 USA}
\email{michaelallocca@muhlenberg.edu}

\author{Max F. Trimmer} 
\address{Department of Mathematics, Computer Science, and Statistics
Muhlenberg College, 
Allentown, PA, 18104 USA}
\email{mtrimmer@muhlenberg.edu}

\begin{document}

\maketitle

\begin{abstract}
We investigate the properties of word lengths of elements from a three-reflection symmetric generating set of the dihedral group $D_n$. Specifically, we provide the upper bound $\lambda_1(D_n,S) \leq \lfloor\frac{n}{2}\rfloor + 1$ for a quantity $\lambda_1$ defined in \cite{MoultonSteel2012}, which measures the stability of a finitely presented group under perturbations in the words corresponding to certain elements with respect to specific presentations. This quantity has been of recent interest due to its role in the application of group theory to computational genomics, and we aim to introduce techniques in additive combinatorics to its discourse.\\

\noindent Keywords: Dihedral Groups, Group Presentations, Word Length, Additive Combinatorics
\end{abstract} 

\section{Introduction}\label{intro}
Given a finitely presented group, $G$, there is a well-established notion of `word length' of any of its elements using a minimal length expression by elements of a generating set.  A symmetric generating set of $G$ is a subset $S \subseteq G$ where any element of $G$ can be written as a product of elements in $S$ (e.g. as a `word'), $e \notin S$, and for all $s_i \in S$, $s_i^{-1} \in S$. For any element $g \in G$, we define the length of $g$ with respect to $S$, $l_S(g)$, to be the minimal number of elements of $S$ (e.g. `letters') for which $g$ can be written as a product of these elements.  The `symmetric' requirement that $s_i^{-1} \in S$ ensures that $l_S(g) = l_S(g^{-1})$.  Further, $l_S(g) = 0 \Leftrightarrow g=e$ and  $l_S(g) = 1 \Leftrightarrow g \in S$.\\  

This notion of word length is often used in the context of a distance metric between two elements of $G$ that is geometrically related to distances among vertices on $\Gamma (G,S)$, the Cayley graph for $G$ associated with the symmetric generating set $S$.  Specifically, for $g,g' \in G$, one may define the distance between $g$ and $g'$ in $ \Gamma (G,S)$, $d_S(g,g')$, as the smallest $k$ such that $g' = gs_1s_2 \cdots s_k$, where $s_i \in S \; \forall\;  i \in \{1,2,\cdots k\}$.  Intuitively, this is the smallest number of edges that may connect the vertices corresponding to $g$ and $g'$ in  $\Gamma (G,S)$ (e.g. the usual `shortest path' metric).  We may use this metric in equivalent word length language by noting that $l_S(g)$ is simply the distance between $g$ and the identity element in the Cayley graph. Further, by the minimality of $k$, $d_S(g,g') =l_S(s_1s_2 \cdots s_k) =  l_S(g^{-1}g')$.  \\

In contemporary literature, these definitions have been useful in computational approaches to genome rearrangements.  For a broad background of how these notions, along with permutation groups and other combinatorial methods, may inform this type of modeling in molecular biology, see \cite{MoultonSteel2012,BafnaPevzner1996,Bergeron2005,EgriNagy2014,Hannenhalli1999}.  Of particular interest is the approach taken in \cite{MoultonSteel2012}, where distance metrics on symmetric groups were used as a tool in modeling ``evolutionary distances'' among organisms based on rearrangements in their respective genomes.  Specifically, if $G=S_n$ is the symmetric group of order $n!$, and $S$ is the usual symmetric generating set consisting of transpositions of adjacent elements (reversals), then distances on the associated Cayley graph can be used to model the shortest number of reversals of gene segments (which are rare evolutionary events) that may connect one organism to another on a phylogenetic tree. In pursuit of these goals, the following was defined in \cite{MoultonSteel2012}:

\begin{align*}
  \lambda_{1}(G,S)&:=  \max_{g\in G,\ s\in S}\{l_{S}(gsg^{-1})\} 
\end{align*}

It is worth noting that $l_{S}(gsg^{-1})= l_{S}((gsg^{-1})^{-1})=l_{S}(g^{-1}s^{-1}(g^{-1})^{-1}) = l_{S}(g^{-1}sg) = d_S(g,sg)$.  So, more intuitively, this quantity measures the effect of a deletion of one letter on word lengths of elements in $G$ which, in \cite{MoultonSteel2012}, served as a way to model the maximum effect a single reversal can have on the distance between the vertices corresponding to two organisms (e.g. the so-called ``butterfly effect'').  Alternatively, one may view this as a measure of `sensitivity' to paths along a Cayley graph following the omission of one generator. This can be viewed as a group theoretic analog to stability in nonlinear dynamical systems. \\

For these reasons, $\lambda_{1}(G,S)$ has been of interest for other choices of $G$ and $S$.  Recently, in \cite{AlloccaEtAl} these notions have been explored in the context where $G = D_n$, the dihedral group of order $2n$, and $S$ is a symmetric generating set with two or three elements.  This investigation required examining all possible symmetric presentations of dihedral groups generated by sets with two or three elements. The latter was a novel classification that explored the relationship between word lengths and whether each of the three elements of $S$ are rotations or reflections. Upper bounds for $\lambda_{1}$ were established for all cases where $S$ has three elements, except when all of these elements are reflections.  In the three-reflection case, where $r$ denotes the usual order $n$ rotation and $f$ is the involution element representing the standard reflection in $D_n$, the following was proved:
\begin{thm}\label{3flip}
Consider the dihedral group $D_{n}$ with $n \geq 2$. Then there exists a generating set of the form $S=\{f,r^{a}f,r^{b}f\}$ such that $\lambda_{1}(D_n,S) \leq \lfloor \frac{n}{2} \rfloor  + 1$, where $a\neq b$ and $a,b \geq1$.
\end{thm}

Although this established the \textit{existence} of a three-reflection generating set for which the stated bound on $\lambda_{1}(G,S)$ is met, in \cite{AlloccaEtAl} it was also conjectured that this serves as a bound for \textbf{all} symmetric generating sets consisting of three reflections as follows.  

\begin{cnj}\label{3flipconj}
For a generating set $S$ composed of three involutions, none of which belong to the order $n$ cyclic subgroup of $D_n$, we conjecture that $\lambda_{1}(D_{n},S) \leq \lfloor \frac{n}{2} \rfloor  + 1$.
\end{cnj}

In this paper, we will prove this conjecture and introduce an alternative strategy of studying $\lambda_{1}(G,S)$ using techniques in additive combinatorics.  We will also show how 3-reflection presentations can be used to write elements of the Dihedral group with shorter word length than is possible for presentations containing two generators.

\section{A bound for prime $n$}
We first aim to establish a smaller bound than the one listed in Conjecture \ref{3flipconj} in the special case where $n$ is prime, subject to a few restrictions on what $S$ can be. Given that $D_n = \{ e, r, r^2, \cdots r^{n-1}, f, rf, r^2f, \cdots r^{n-1}f\}$, we adopt the standard nomenclature in considering the $n$ elements of the form $r^i$ to be all of the rotations and the $n$ elements of the form $r^if$ to be all of the reflections, where the $n$ rotations form a cyclic subgroup generated by $r$. Further information can be found in \cite{DummitFoote2004} regarding dihedral groups, which can also be seen as a special class of Coxeter groups \cite{Humphreys1990}. Unless specified otherwise, $S= \{f,r^af,r^bf\}$.  In \cite{AlloccaEtAl}, it was shown that if $H_1=\left\langle r^a \right\rangle$ and $H_2=\left\langle r^b \right\rangle$, then $S$ generates $D_{n}$ if and only if $H_{1}H_{2} = \{h_{1}h_{2}\mid \ h_{i}\in H_{i}\} = \langle r\rangle = \{1,r,\ldots,r^{n-1}\}$.  Further, $S$ is equivalent to the generating set for any three-reflection presentation, $S' = \{r^if',r^jf',r^kf'\}$, by simply labeling the first reflection in the set $f$.  That is, it will suffice to consider $S$ as defined above because the results in this paper will apply to $S'$ by letting $f=r^if', a = j-i,$ and $b = k-i$.\\

Extending the notation introduced in Section \ref{intro}, let $W_l$ be the set of elements of $D_n$ that can be written (not necessarily minimally) as a product of $l$ elements of $S$.  Note that if $l_S(g) = k$, then $g \in W_k$. However the converse need not be true, as a word expressed as a product of $k$ letters in $S$ need not be a word of length $k$. It is also worth noting that if $g \in W_k$, then $l_S(g) \leq k$. For example, if $g=r^af \cdot f \cdot f$, then $g \in W_3$ and its length is no greater than $3$. However, $l_s(g) =1$ since $g=r^a \cdot e = r^a$ is its minimal expression (and thus $g \in W_1$ as well). 



\begin{lma}
An element of $D_n$ can be written as a product of an odd number of elements in $S$ if and only if it is a reflection. An element of $D_n$ can be written as a product of an even number of elements in $S$ if and only if it is a rotation.
\end{lma}
\begin{proof}
Let $x \in D_n$ and suppose it can be written as a product of $l$ elements in $S$. Suppose $l$ is odd.  Then $x$ may be written as a product of $l-1$ reflections and a single reflection.  Since $l-1$ is even, the latter product is a rotation.  So $x$ is the product of a rotation and a reflection, which is a reflection.  Now suppose $x$ is not a reflection.  Then it is a rotation.  Since $S$ consists entirely of reflections, if $x \neq e$, then $l$ must be even, otherwise $x$ would be a rotation.  Therefore, by contraposition, if $x$ is a reflection, then it can be written as a product of an odd number of elements in $S$. \\

Hence, $x$ can be written as a product of an odd number of elements in $S$ if and only if it is a reflection.  Similarly, $x$ can be written as a product of an even number of elements in $S$ if and only if it is a rotation.
\end{proof}

Note that this implies that an element cannot be written as a product of both an even and an odd number of letters, as that would imply that the element is both a rotation and a reflection. \\

We will adopt sumset notation following \cite{TaoVu}.  For two subsets $A$ and $B$ of an additive group $(G,+)$, we define $A + B = \{a + b \;|\; a\in A,b\in B\}$ and $A - B = \{a + (-b) \; |\; a\in A,b \in B\}$ where $-b$ denotes the inverse of $b$ in $G$. Further, for subsets $A^{'}$ and $B^{'}$ of a multiplicative group $(G^{'}, \cdot)$, we define $A^{'}B^{'} = \{a\cdot b \; |\; a\in A^{'}, b \in B^{'}\}$.  Unless stated otherwise, addition will take place in a cyclic group and multiplication will take place in a dihedral group. Note that for $l > 0$, $W_l = W_{l-1}S$, since any word expressed as a product of $l$ elements in $S$ is also expressed as a product of $l-1$ elements with a letter from $S$ appended to it, and a word expressed as a product of $l-1$ elements with a letter from $S$ appended to it is also expressed as a product of $l$ elements from $S$.\\ 
 
We can exploit the fact that each $W_l$ exclusively consists of rotations or reflections as follows: Let $F: D_n \rightarrow \mathbb{Z}_n$ be defined by $F(r^i f^k) = i$. Note that $F$ is not one-to-one, as the images of the rotation subgroup and the reflection coset under $F$ are both $\mathbb{Z}_n$. However, since $W_l$ either contains only reflections or only rotations, $|F(W_l)| = |W_l|$ for any $l$, where $F(W_l)$ denotes the image of $W_l$ under $F$. Similarly, since $S= \{f,r^af,r^bf\}$ only contains reflections, $|F(S)| = |S| = 3$, as $F$ simply excises the number of times an element has been `rotated'. From here on, we will denote $F(W_l)$ as $W_l'$ and $F(S)$ as $S'$. Note that $W_{l}' \subseteq W_{l + 1}'$ for any $l \geq 0$ since $0 \in S' = \{0,a,b\}$, so for any element $x \in W_{l}'$, $x + 0 = x \in W_{l+1}'$. \\

\begin{lma}\label{paritysumdiff}
If $l$ is even, then $W_{l+1}' = W_l'+S'$. If $l$ is odd, then $W_{l+1}' = W_l' - S'$ 
\end{lma}

\begin{proof}
By Lemma 1, if $l$ is even, then $W_l$ is a subset of the rotation subgroup. Therefore, every element is of the form $r^x$ for some $0 \leq x < n$. Further, $W_{l+1} = W_lS = \{r^x \cdot r^kf \;|\; r^x\in W_l, r^kf\in S\} = \{r^{x+k}f \;|\; r^x\in W_l, r^kf\in S\}$. Therefore, $W_{l+1}' = \{F(r^{x+k}f) \;|\; x \in W_l', k \in S'\} = \{x+k \;|\; x\in W_l',k\in S'\} = W_l' + S'$. \\

Similarly, by Lemma 1, if $l$ is odd, then $W_l$ is a set of reflections. Therefore, every element is of the form $r^kf$ for some $0 \leq x < n$. $W_{l+1} = W_lS = \{r^xf \cdot r^kf \;|\; r^x\in W_l, r^kf\in S\} = \{r^{x-k} \;|\; r^x\in W_l, r^kf\in S\}$. Therefore, $W_{l+1}' = \{F(r^{x-k}) \;|\; x \in W_l', k \in S'\} = \{x-k \;|\; x\in W_l',k\in S'\} = W_l' - S'$.
\end{proof}

One advantage of examining all of this in the context of sumsets involves the following theorem, which was originally proven by Augustin-Louis Cauchy \cite{cauchy}, and independently discovered by Harold Davenport \cite{dav1}, who was, at the time, unaware of Cauchy's historial contributions in this context \cite{dav2}.  

\begin{thm}[Cauchy-Davenport]
Given a cyclic group $\mathbb{Z}_p$ of prime order $p$ and two subsets $A$ and $B$, $|A+B| \geq \text{min}(p,|A|+|B|-1)$ 
\end{thm}

We may invoke this as a tool to bound $W_l$ with a prime modulus.  

\begin{lma}\label{prime}
Let $p$ be prime and $l \geq 0$. If $S = \{f,r^af,r^bf\}$, and $0,a,b,-a,-b,a-b,b-a$ are all distinct modulo $p$, then $|W_l| \geq min(3l,p)$ with respect to $S$. 
\end{lma}

\begin{proof}
We first note that $S' - S' = \{0,a,b,a-b,b-a,-a,-b\}$. Since we have assumed these are all distinct modulo $p$, $|S'-S'| = 7$. We will show that $|W_l| \geq min(3l,p)$ for odd $l$ by induction on $l$. First note that $|W_0| = 1 \geq 3*0=0$ since $W_0 = \{r^0\}$ and that $|W_1| = 3 \geq 3*1 = 3$ since $W_1 = S$. Assume that $|W_{l-2}| \geq min(3(l-2),p) = min(3l - 6,p)$, with $l$ an odd number greater than 1. By the Cauchy-Davenport theorem and Lemma \ref{paritysumdiff}, $|W_l| = |W_l'| = |W'_{l-1} +  S| = |W'_{l-2} - S + S| = |W'_{l - 2} - (S - S)| \geq  |W'_{l-2}| + |S-S| - 1 \geq min(3l-6,p) + 7 - 1 = min(3l,p)$. Likewise, if $l$ is even, then $|W_l| = |W_l'| = |W'_{l-1} -  S| = |W'_{l-2} + S - S| \geq  |W'_{l-2}| + |S-S| - 1 \geq min(3l-6,p) + 7 - 1 = min(3l,p)$
\end{proof}

\begin{crl}\label{prime_bound}
Let $p$ be prime. If $S = \{f,r^af,r^bf\}$, and $0,a,b,-a,-b,a-b,b-a$ are all distinct modulo $p$, then $\lambda_1(D_p,S) \leq \lfloor\frac{p}{3}\rfloor+1$ 
\end{crl}


\begin{proof}
First note that since all elements of the form $gsg^{-1}$ are reflections, $\lambda_1$ is less than or equal to the maximum word length of a reflection in $D_p$ using elements of $S$.  Also recall that if $x \in W_k$, then its length is no greater than $k$.  Therefore, it suffices to show that every reflection in $D_p$ is an element of $W_{\lfloor \frac{p}{3}\rfloor}$ or $W_{\lfloor \frac{p}{3}\rfloor + 1}$.  By Lemma \ref{prime}, $|W_l| \geq min(3l,p)$, and since there are $p$ reflections and $W_l$ is composed entirely of reflections for odd $l$, this implies that $W_l$ contains all the reflections in $D_p$ by the pigeonhole principle. This happens when $l > p/3$. Since $l$ must be odd to get reflections, we get the lower bound $\lambda_1 \leq \lfloor\frac{p}{3}\rfloor + 1$
\end{proof}


\section{A bound for all $n$}
We now aim to provide a full proof of Conjecture \ref{3flipconj} for arbitrary $n$. We define the stabilizer of a set $A$ in an abelian group $(G,+)$ to be $Stab(A)=\{g\in G \, | \, g + A = A\}$. This is defined in \cite{TaoVu} as a symmetry group, although we use the term stabilizer to remain consistent with the nomenclature commonly connected to the following, which first appeared in \cite{kneser}. 

\begin{thm} [Kneser's]  
Given an Abelian group $G$ and two sets of elements $A$ and $B$ such that $|A + B| < |A| + |B|$, $|A+B| = |A + \text{Stab(A+B)}| + |B+\text{Stab(A+B)}| - |\text{Stab(A+B)}|$
\end{thm}

Kneser's theorem allows us to bound $|W_l|$ in a useful manner. 

\begin{thm}\label{minBound}
$|W_l| \geq min(2l + 1,n)$ for $l > 0$
\end{thm}

\begin{proof}
We proceed by induction on $l$. $|W_1| = 3 \geq 2*1 + 1$. Assume that $W_{l-1} \geq min(2(l - 1) + 1,n)$. If $l$ is odd, then $l-1$ is even, so by Lemma \ref{paritysumdiff}, $|W_l| = |W_l'| = |W_{l-1}' + S'|$. Given $S' = \{0,a,b\}$, there are 3 cases based on whether $a$ and $b$ are elements of $\text{Stab}(W_{l}')$\\ \\
\textit{Case 1}: $a,b \in \text{Stab}(W_{l}')$ \\ \\
Since $a,b \in \text{Stab}(W_{l}')$, and $\text{Stab}(W_{l}')$ is a group, $-a$ and $-b$ are also in $\text{Stab}(W_{l}')$. Since moving from $W_k'$ to $W_{k + 1}'$ involves either adding or subtracting $S$, and since every element of $S$ is in $\text{Stab}(W_{l+1}')$ $W_k' = W_l'$ for any $k \geq l$. Since $|W_l'| = |W_l|$, this means that either $|W_l| = n$, or $S$ does not generate $D_n$. 
 \\ \\ \textit{Case 2}: $\text{Stab}(W_{l}') = \{0\}$ \\ \\
By Kneser's theorem, either $|W_l'| \geq |W_{l-1}'| + |S'| \geq 2l - 1 + 3 = 2l+2$, or $|W_l'| = |W_{l-1}' + \text{Stab}(W_{l}')| + |S' + \text{Stab}(W_{l}')| - |\text{Stab}(W_{l}')| = |W_{l-1}'| + |S| - 1 \geq 2l+1$. \\ \\
\textit{Case 3}: Neither $a$ nor $b$ are elements of $\text{Stab}(W_{l}')$, but $\text{Stab}(W_{l}') \neq \{0\}$ \\ \\
Let $c$ be a non-identity element of $\text{Stab}(W_{l}')$. Since $c$ is in the stabilizer, $W_l$ is a collection of cosets of $\langle c \rangle$. Since $c$ is not the identity, each coset of $\langle c \rangle$ has at least 2 elements. \\ \\ Assume that $|W_l'| - |W_{l-1}'| < 2$. (Otherwise $|W_l| = |W_l'| \geq |W_{l-1}'| + 2 = |W_{l-1}|+2 \geq min(2(l - 1) + 1,n) + 2 \geq min(2l + 1,n)$). If $|W_l'| = |W'_{l - 1}|$, then since $W_{l-1}' \subseteq W_{l}'$, $W_l' = W_{l-1}'$. We know that $W'_{l-1}+a \subseteq W_l'$ and $W'_{l-1} \subseteq W'_{l-1} + a$. Therefore since $|W'_l| = |W'_{l+1}|$, $|W'_{l-1}+a| = |W'_{l}|$. This would imply that $W'_{l-1} + a = W'_l$, contradicting the assumption that $a$ is not in the stabilizer. \\ \\ The only sub-case left is if $|W_l'| = |W'_{l-1}|+1$. Because $W'_{l-1} \subseteq W'_l$ and $c \in \text{Stab}(W'_l)$, $W_l'$ is a union of cosets of $\langle c \rangle$ and $W'_{l-1}$ is a union of cosets of $\langle c \rangle$ missing some element $x$. Since $a$ is not a stabilizer of $W'_{l}$, there must be an element of $W'_{l}+a$ not in $W'_l$. However, since $W'_{l} + a = (W'_{l-1} + a)\cup (W'_l/W'_{l-1} + a) = (W'_{l-1} + a)\cup \{x + a\}$, this element has to be $x+a$. Since $x + a \notin W'_{l}$ and $c \in H$, the coset $x + a + \langle c \rangle$ is disjoint from $W'_{l}$. However, since $c$ is not the identity, there is another element $x' \in W_{l-1}$ in the same coset as $x$. Therefore, $x' + a \in W_l$, a contradiction. \\ \\
\textit{Case 4}: $a \in \text{Stab}(W_{l}')$ or $b \in \text{Stab}(W_{l}')$, but not both. \\ \\ Without loss of generality, assume $a \in \text{Stab}(W_{l}')$ and $b$ is not. For the same reason as the previous case, we can assume $|W'_{l}| - |W'_{l-1}| < 2$. If $|W'_{l}| = |W'_{l-1}|$, then since $W'_l \subseteq W'_{l-1}$, $W'_l = W'_{l-1}$. This implies that $W'_{l-1} + b \subseteq W'_{l-1}$ since $W'_{l-1} + a \subseteq W'_{l-1}$ and $W'_l + 0 = W'_l$. Therefore, $b \in \text{Stab}(W'_{l-1}) = \text{Stab}(W'_l)$, a contradiction. If $|W'_{l} - W'_{l - 1}| = 1$, then $W'_{l-1}$ is a collection of cosets of $\langle a \rangle$ missing one element $x$. $W'_{l-1} + b \subseteq W'_l$. Let $x'$ be an element in the same coset as $x$. Since $x' \in W'_l$ and $a \in \text{Stab}(W'_l)$, $x' \in W'_{l+1}$. Since $W'_{l+1} = W_l + \{x\}$ and $W'_{l-1} + b \subseteq W'_l$, $x' + b \in W'_{l}$. Since $x + b$ and $x' + b$ are in the same coset of $\langle a \rangle$ and $a \in \text{Stab}(W'_l)$, $x + b \in W_{l - 1}$. Therefore, either $b \in \text{Stab}(W'_{l})$, a contradiction, or there is an element in $W'_l \backslash W'_{l-1}$ in a different coset than $x$, also a contradiction. Therefore, $|W_l| = |W'_l| \geq min(|W_{l-1}| + 2,n) \geq min(2l+1,n)$
\\ \\
The cases where $l$ is even are the same, but with $a$ replaced by $-a$ and $b$ replaced by $-b$.
\end{proof}
\begin{crl}
For a three-reflection presentation of the dihedral group,
$\lambda_1(D_n,S) \leq \lfloor\frac{n}{2}\rfloor + 1$ 
\end{crl}
\begin{proof}
 With justifications similar to those in the proof of Corollary \ref{prime_bound}, it suffices to show that every reflection in $D_n$ is an element of $W_{\lfloor\frac{n}{2}\rfloor + 1}$ or $W_{\lfloor\frac{n}{2}\rfloor}$.  By Theorem \ref{minBound}, for any three-reflection presentation of $D_n$, $|W_{\lceil(n-1)/2\rceil}| \geq min(2*\lceil\frac{n-1}{2}\rceil + 1,n) \geq min(n,n) = n$. If $\frac{n-1}{2}$ is odd, then $W_{(n-1)/2}$ is composed entirely of reflections. Since there are $n$ reflections in $D_n$, every reflection is in $W_{(n-1)/2}$ by the pigeonhole principle. Similarly, if $\lceil\frac{n-1}{2}\rceil$ is even, then since $W_{\lceil(n-1)/2\rceil+1}$ is composed entirely of reflections, and $|W_{\lceil(n-1)/2\rceil+1}| \geq n$, every reflection is in $W_{\lceil(n-1)/2\rceil+1}$ by the pigeonhole principle. Therefore, every reflection in $D_n$ can be written as a product $\lceil\frac{n-1}{2}\rceil+1=\lfloor\frac{n}{2}\rfloor+1$ elements of $S$. (Note that the last equality holds since $n$ is required to be an integer).
\end{proof}

The proof of Theorem \ref{minBound} suggests that the following is likely true. 
\begin{cnj}
The stabilizer of $W_l'$ is always either $\{1\}$ or $C_n$
\end{cnj}

\section{Bound Sharpness}
 We now show that the bound is very close to sharp. 

\begin{thm} Given a three-reflection symmetric generating set, $S$, of $D_n$, 
$$\underset{a \in \mathbb{N}}{\max} \; l_S(r^af)) = \begin{cases}
\lfloor\frac{x}{2}\rfloor + 1 & \text{if } x \equiv 0 \:\text{mod} \, 4,\\
  \lfloor\frac{x}{2}\rfloor + 1 & \text{if } x \equiv 1 \:\text{mod} \, 4,\\
 \lfloor\frac{x}{2}\rfloor & \text{if } x \equiv 2 \:\text{mod} \,4, \\
 \lfloor\frac{x}{2}\rfloor& \text{if } x \equiv 3 \:\text{mod} \, 4,
\end{cases}$$
\end{thm}
\begin{proof}
First note that $W_1' = \{-1,0,1\}$. We will show by induction that for any $n$ and $k > 1$ $W_k' = \{-k,\ldots,0,\ldots,k\} \:\text{or}\:\mathbb{Z}_n$. Assume that $W_{k-1}' = \{-(k-1),\ldots,0,\ldots,k-1\}\:\text{or}\:\mathbb{Z}_n$. If $W_{k-1}' = \mathbb{Z}_n$, then $W_k' = \mathbb{Z}_n$ since $0 \in S'$. Adding $S = \{-1,0,1\}$ to $W_{k-1}' = $ gives $\{-k,\ldots,0,\ldots,k\}$ To see this, first note that $W'_{k-1} \subseteq W'_{k}$ since $0 \in S'$. $k$ and $-k$ are also in $W_k'$ since $k = k-1 + 1$ and $-k = (-k-1) - 1$. There are no other elements since the absolute value of any element in $W'_{k-1}$ is $k-1$ or less and the absolute value of anything in $S'$ is 1 or less. Since $|W'_{l}| = \text{min}(2l+1,n)$, $W_{l}$ contains every reflection only when $2l+1 \geq n$ and $l$ is odd, which first happens when $$l = \begin{cases}
\lfloor\frac{x}{2}\rfloor + 1 & \text{if } x \equiv 0 \:\text{mod} \, 4,\\
  \lfloor\frac{x}{2}\rfloor + 1 & \text{if } x \equiv 1 \:\text{mod} \, 4,\\
 \lfloor\frac{x}{2}\rfloor & \text{if } x \equiv 2 \:\text{mod} \,4, \\
 \lfloor\frac{x}{2}\rfloor& \text{if } x \equiv 3 \:\text{mod} \, 4,
\end{cases}$$
\end{proof}

We now shift focus to finding presentations where every reflection can be written using as few words as possible. \\
\begin{thm}
Given the Dihedral group $D_n$, there is a 3-reflection presentation where every element has word length $O(\sqrt{n})$
\end{thm}
\begin{proof}
Let $S = \{f,rf,r^{\lfloor\sqrt{n}\rfloor}f\}$. \\ For any reflection $r^af$, let $b$ and $s$ be the unique constants such that $s < {\lfloor\sqrt{n}\rfloor}$, and $a = b{\lfloor\sqrt{n}\rfloor} + s$. $r^af$ can be written as $(r^{\lfloor\sqrt{n}\rfloor}ff)^b(rff)^s$, so $l_w(r^af) \leq 2r+2b$

\end{proof}
\section{Conclusion}
The approaches adopted in the previous sections provide further insights for minimal word length calculations.  We believe that tools in additive combinatorics can be of further use in analyzing the properties of $\lambda_1$ with different symmetric generating sets and different groups.

\bibliographystyle{elsarticle-num}
\bibliography{refs}
\end{document}